\documentclass[11pt,a4, reqno]{amsart}
\usepackage{amsmath}
\usepackage{graphicx}
\newtheorem{defi}{Definition}
\newcommand{\brdef}{\begin{defi}}
\newcommand{\erdef}{\end{defi}}
\newtheorem{cor}{Corollary}
\newcommand{\bcor}{\begin{cor}}
\newcommand{\ecor}{\end{cor}}

\newtheorem{thm}{Theorem}
\newcommand{\bth}{\begin{thm}}
\newcommand{\eth}{\end{thm}}
\newtheorem{lem}{Lemma}

\newcommand{\ble}{\begin{lem}}
\newcommand{\ele}{\end{lem}}

 \huge

\usepackage{hyperref}

\newtheorem{rem}[thm]{Remark}
\numberwithin{equation}{section}

\begin{document}
\begin{center}
{\Large \bf Ricci Soliton and Geometrical Structure in a Perfect Fluid Spacetime with Torse-forming Vector Field}\\
\
\\
{ Venkatesha{\footnote {Corresponding Author}} and Aruna Kumara H}\\
\end{center}
\begin{quotation}
{\bf Abstract:} In this paper geometrical aspects of  perfect fluid spacetime with torse-forming vector field $\xi$ are discribed and Ricci soliton in perfect fluid spacetime with torse-forming vector field $\xi$ are determined. Conditions for the Ricci soliton to be expanding, steady or shrinking
are also given. 
 Energy momentum tensor. Lorentz space. Ricci soliton.
\end{quotation}

\section{\bf Introduction}
The spacetime of general relativity and cosmology can be modeled as a connected $4-$dimensional pseudo-Riemannian manifold with Lorentzian metric $g$ with signature $(-,+,+,+)$. The geometry of Lorentzian manifold begins with the study of causal character of vectors of the manifold, due to this causality that Lorentzian manifold becomes a convenient choice for the study of general relativity. 
\par In the general theory of relativity, energy-momentum tensor plays an important  role and the condition on the energy-momentum tensor for a perfect fluid spacetime changes the nature of spacetime\cite{SH}. A matter content of the spacetime is described by the energy-momentum tensor, matter is assumed to be a fluid having density, pressure and possessing dynamical and kinematical quantities like velocity, acceleration, vorticity, shear and expansion. Since the matter content of the
universe is assumed to behave like a perfect fluid in the standard cosmological models.
\par A perfect fluid is a fluid that can be completely characterized by its rest frame mass density and isotropic pressure. Perfect fluids are often used in general relativity to model idealized distributions of matter, such as the interior of a star or an isotropic universe. It has no shear stress, viscosity, or heat conduction and is characterized by an energy-momentum
tensor $T$ of the form\cite{BO}
\begin{align}
\label{1.1} T(X,Y)=\rho g(X,Y)+(\sigma+\rho)\eta(X)\eta(Y),
\end{align}
where $\rho$ is the isotropic pressure, $\sigma$ is the energy-density, $\eta(X)=g(X,\xi)$ is $1-$form, equivalent to unit vector $\xi$ and $g(\xi,\xi)=-1$. The field equation governing the perfect fluid motion is Einstein's gravitational equation\cite{BO}
\begin{align}
\label{1.2} S(X,Y)+\left(\lambda-\frac{r}{2}\right)g(X,Y)=\kappa T(X,Y),
\end{align}

where $S$ is Ricci tensor, $r$ is scalar curvature of $g$, $\lambda$ is the cosmological constant and $\kappa$ is the gravitational constant (which can be taken $8\pi G$, with $G$ as universal gravitational constant). According to Einstein, equation \eqref{1.2} obtained from Einstein equation by adding cosmological constant to get static universe. In modern cosmology, it is considered as a candidate for dark energy, the cause of the acceleration of the expansion of the universe.
\par In view of \eqref{1.1}, equation \eqref{1.2} takes the form
\begin{align}
\label{1.3} S(X,Y)=\left(-\lambda+\frac{r}{2}+\kappa\rho\right)g(X,Y)+\kappa(\sigma+\rho)\eta(X)\eta(Y).
\end{align}
Suppose Ricci tensor $S$ is a functional combination of $g$ and $\eta\otimes\eta$, called quasi-Einstein\cite{CH}. Quasi-Einstein manifolds arose during the study of exact solutions of Einstein field equation. For example, the Robertson-Walker spacetime are quasi-Einstein manifolds\cite{Des}. They also can be taken as a model of the perfect fluid spacetime in general relativity (\cite{UCD}, \cite{UCDe}). Perfect fluid spacetime are extensively studied in many purpose of view; we may refer to (\cite{ZA},\cite{CHM},\cite{DEU},\cite{SG},\cite{SM},\cite{SR}) and references therein. 
\par Ricci flows are intrinsec geometric flows on a pseudo-Riemannian
manifold, whose fixed points are solitons, it was introduced by Hamilton\cite{HAM}. Ricci solitons also correspond to
self- similar solutions of Hamilton's Ricci flow. They are natural generalization of Einstein metrics
and is defined by
\begin{align}
\label{1.4}(L_V g)(X,Y)+2S(X,Y)+2\Lambda g(X,Y)=0,
\end{align}
for some constant $\Lambda$, a potential vector field $V$. The Ricci soliton is said to be shrinking, steady, and expanding according as $\Lambda$ is negative, zero, and positive respectively. In the papers(\cite{BEJ}-\cite{CAL},\cite{RS},\cite{RSB}), Ricci solitons are studied extensively within the background of psuedo-Riemannian geometry.

In this paper, we are interested to study some geometrical aspect and Ricci soliton in perfect fluid spacetime with torse-forming vector field $\xi$.

\section{\bf Basic Properties of Perfect Fluid Spacetime with Torse-forming Vector Field}
Let $(M^4,g)$ be a relativistic perfect fluid spacetime satisfying \eqref{1.3}. Contracting \eqref{1.3} and taking under consideration that $g(\xi,\xi)=-1$, we get
\begin{align}
\label{2.1} r=4\lambda+\kappa(\sigma-3\rho).
\end{align}
Therefore
\begin{align}
\label{2.2} S(X,Y)=ag(X,Y)+b\eta(X)\eta(Y),
\end{align}
where $a=\lambda+\frac{\kappa(\sigma-\rho)}{2}$, and $b=\kappa(\sigma+\rho)$, or equivalently,
\begin{align}
\label{2.3}QX=aX+b\eta(X)\xi.
\end{align}

We shall treat $\xi$ as a torse-forming vector field. According to general definition(\cite{AMB},\cite{YAK}), it is a covariant derivative satisfying
\begin{align}
\label{2.4}\nabla_X\xi=X+\eta(X)\xi.
\end{align} 
\begin{thm}
	On perfect fluid spacetime with torse-forming vector field $\xi$, the following relations hold;
	\begin{align}
	\label{2.5}(\nabla_X\eta)(Y)&=g(X,Y)+\eta(X)\eta(Y),\\
	\label{2.6} \eta(\nabla_X \xi)&=0, \qquad \nabla_\xi \xi=0,\\
	\label{2.7}R(X, Y)\xi&=\eta(Y)X-\eta(X)Y,\\
	\label{2.8} R(X,\xi)\xi&=-X-\eta(X)\xi,\\
	\label{2.9}\eta(R(X,Y)Z)&=\eta(X)g(Y,Z)-\eta(Y)g(X,Z),\\
	\label{2.10}(L_\xi g)(X,Y)&=2[g(X,Y)+\eta(X)\eta(Y)],
	\end{align}	
\end{thm}
\begin{proof}
	Compute $(\nabla_X \eta)(Y)=X(\eta(Y))-\eta(\nabla_X Y)=X(g(Y,\xi))-g(\nabla_X Y,\xi)=g(Y,\nabla_X \xi)=g(X,Y)+\eta(X)\eta(Y).$ In particular $(\nabla_\xi \eta)(Y)=0$. The relation \eqref{2.6} can be obtained by \eqref{2.4}.
	\par Substitute the expression of $\nabla_X \xi$ in $R(X,Y)\xi=\nabla_X\nabla_Y \xi-\nabla_Y\nabla_X \xi-\nabla_{[X,Y]} \xi$ and from direct computation we have a tendency to get the relation \eqref{2.7}. Additionally \eqref{2.8} and \eqref{2.9} follows from \eqref{2.7}.
	\par Now Lie differentiating $g$ along $\xi$, followed by simple calculation we get \eqref{2.10} 
\end{proof}
\begin{thm}
	The following conditions hold in perfect fluid spacetime with torse-forming vector field $\xi$;
	\begin{align}
	\label{2.11} g(\nabla_X \xi, \nabla_Y \xi)=-g(R(X,\xi)\xi,Y)=-g(R(\xi,Y)X,\xi),\\
	\label{2.12} g(R(X,\xi)\xi,Y)=\frac{1}{2}\{g(R(X,\xi)\xi,\nabla_Y \xi)+g(R(Y,\xi)\xi,\nabla_X \xi)\},\\
	\label{2.13} g(\nabla_\xi\nabla_Y \xi, \nabla_X \xi)+g(R(X,\xi)\xi,\nabla_\xi Y)=0.
	\end{align}
\end{thm}
\begin{proof}
	From \eqref{2.4} and \eqref{2.8}, it follows equation \eqref{2.11} and \eqref{2.12}. Now consider
	\begin{align*}
	g(\nabla_\xi\nabla_Y \xi, \nabla_X \xi)=g(\nabla_{\nabla_\xi Y}\xi, \nabla_X \xi)=-g(R(X,\xi)\xi,\nabla_\xi Y).
	\end{align*}
	This gives the condition \eqref{2.13}. This completes the proof
\end{proof}

\section{\bf Some Geometric Properties of Perfect Fluid Spacetime with Torse-forming Vector Field}
A perfect fluid spacetime with torse-forming vector field of dimension $4$ is said to be conformally flat, if the Weyl conformal curvature tensor $C$ vanishes and is defined by\cite{YKM}

\begin{align}
\label{3.1}\nonumber C(X,Y)Z&=R(X,Y)Z-\frac{1}{2}[S(Y,Z)X-S(X,Z)Y+g(Y,Z)QX-g(X,Z)QY]\\
&+\frac{r}{6}[g(Y,Z)X-g(X,Z)Y]
\end{align} 
\begin{thm}\label{t3.2}
	If perfect fluid spacetime with torse-forming vector field $\xi$ is conformally flat, then the energy-momentum tensor is Lorentz-invariant and is of constant curvature $(1/3)(\lambda+\kappa\sigma)$.
\end{thm}
\begin{proof}
	Let $(M^4,g)$ be a conformally flat perfect fluid spacetime with torse forming vector field $\xi$. As $C=0$, we have $div C=0$, where $div$ denotes the divergence. Hence, from \eqref{3.1} we have
	\begin{align*}
	(\nabla_X S)(Y,Z)-(\nabla_Y S)(X,Z)=\frac{1}{6}\left[X(r)g(Y,Z)-Y(r)g(X,Z)\right].
	\end{align*} 
	Which is equivalent to
	\begin{align}
	\label{3.2}	g((\nabla_X Q)Y-(\nabla_Y Q)X, Z)=\frac{1}{6}[X(r)g(Y,Z)-Y(r)g(X,Z)].
	\end{align}
	Since scalar curvature $r$ is constant and from \eqref{2.3}, equation \eqref{3.2} leads to
	\begin{align}
	0=(\nabla_X Q)Y-(\nabla_Y Q)X=b[(\nabla_X\eta)(Y)\xi+\eta(Y)\nabla_X \xi-(\nabla_Y \eta)(X)\xi-\eta(Y)\nabla_X \xi]
	\end{align}
	Then from \eqref{2.4} and \eqref{2.5}, it follows that
	\begin{align*}
	b[g(Y,\xi)X-g(X,\xi)Y]=0,
	\end{align*}
	which shows that $b=0$, implies that $\sigma=-\rho$, the energy-momentum tensor is Lorentz-invariant and during this case we talk about the vacuum. 
	\par From \eqref{2.3}, we have $QX=aX$. So $C=0$ implies 
	\begin{align}
	\label{3.4}	R(X,Y)Z=\frac{1}{3}(\lambda+\kappa\sigma)\{g(Y,Z)X-g(X,Z)Y\},
	\end{align}
	which means $(M^4,g)$ is of constant curvature $(1/3)(\lambda+\kappa\sigma)$. This completes the proof.
\end{proof}	
A pseudo-Riemannian manifold is said to be of quasi constant curvature if the curvature tensor $R$ of type $(0,4)$ satisfies
\begin{align}
\nonumber R(X,Y,Z,W)=&m\{g(Y,Z)g(X,W)-g(X,Z)g(Y,W)\}\\
\nonumber &+n\{g(X,W)A(Y)A(Z)-g(X,Z)A(Y)A(W)\\
\label{3.5}&+g(Y,Z)A(X)A(W)-g(Y,W)A(X)A(W)\},
\end{align}
where $m$ and $n$ are scalars and $A$ could be a non-zero 1-form such $g(X,U)=A(X)$ for all X, $U$ being a unit vector field. The notion of a manifold of quasi constant curvature was introduced by Chen and Yano\cite{CHY}, in 1972.
\begin{rem}
	From \eqref{3.4} it follows that a conformally flat perfect fluid spacetime with torse-forming vector field $\xi$ is of quasi constant curvature with $m=(1/3)(\lambda+\kappa\rho)$ and $n=0$ in \eqref{3.5}.
\end{rem}
We know that manifold of costant curvature is Einstein. From Theorem \ref{t3.2}, we state that;
\begin{thm}
	A conformally flat perfect fluid spacetime with torse-forming vector field $\xi$ is an Einstein.
\end{thm}
If a pseudo-Riemannian manifold satisfies the condition $R(X,Y)\cdot R=0$ and $R(X,Y)\cdot S=0$, then it is called semi-symmetric and Ricci semi-symmetric respectively. The condition $R(X,Y)\cdot R=0$ implies $R(X,Y)\cdot S=0$, but the converse need not holds true.  \\
Now, we prove the following;
\begin{thm}
	A conformally flat perfet fluid spacetime with torse-forming vector field $\xi$ is semi-symmetric and Ricci semi-symmetric.
\end{thm}
\begin{proof}
	From, \eqref{3.4}, we can easily show that $R(X,Y)\cdot R=0$ and this condition implies $R(X,Y)\cdot S=0$. This completes the proof.
\end{proof} 
\begin{defi}
	A second order tensor $\alpha$ is called a parallel tensor if $\nabla \alpha=0$, where $\nabla$ denotes the operator of covariant differentiation with respect to the metric tensor $g$.
\end{defi} 
Now, we prove the following result;
\begin{thm}
	\label{t3.7}
	If a conformally flat perfect fluid spacetime with torse-forming vector field $\xi$ admits a second order symmetric parallel tensor, then either $\lambda=-\kappa\sigma$, or the second order parallel tensor is a constant multiple of metric tensor $g$.
\end{thm} 

\begin{proof}
	Let $\alpha$ be a second order symmetric tensor which will considered to be parallel with respect to $\nabla$ i.e., $\nabla \alpha=0$.
	\\ Applying the Ricci commutation identity
	\begin{align*}
	\nabla^2_{X, Y}\alpha(Z,W)-\nabla^2_{X,Y}\alpha(W,Z),
	\end{align*}
	we obtain the following fundamental relation;
	\begin{align}
	\alpha(R(X,Y)Z,W)+\alpha(Z,R(X,Y)W)=0.
	\end{align}
	Subtituting $Z=W=\xi$ in above relation, by the symmetry of $\alpha$ and from \eqref{3.4} it results 
	\begin{align}
	\frac{2}{3}(\lambda+\kappa\sigma)[\eta(Y)\alpha(X,\xi)-\eta(X)\alpha(Y,\xi)]=0,
	\end{align}
	showing that either $\lambda=-\kappa\sigma$ (so $\sigma=-\lambda/\kappa$) or $\eta(Y)\alpha(X,\xi)-\eta(X)\alpha(Y,\xi)=0$.\\
	Consider
	\begin{align*}
	\eta(Y)\alpha(X,\xi)-\eta(X)\alpha(Y,\xi)=0.
	\end{align*}
	Replacing $X=\xi$, in the above equation, it follows that
	\begin{align}
	\label{3.8}\alpha(Y,\xi)=-\eta(Y)\alpha(\xi,\xi).
	\end{align}
	The parallelism of $\alpha$ and equation \eqref{3.8} imply that $\alpha(\xi,\xi)$ is constant:
	\begin{align}
	\label{3.9}X(\alpha(\xi,\xi))=2\alpha(\nabla_X \xi,\xi)=2\eta(\nabla_X \xi)\alpha(\xi,\xi)=0.
	\end{align}
Differentiating \eqref{3.8} along $X$ and using \eqref{3.9}, we have 
	\begin{align}
	\label{3.10}X(\alpha(Y,\xi))=-X(g(Y,\xi))\alpha(\xi,\xi)=-[g(\nabla_X Y,\xi)+g(Y,\nabla_X \xi)]\alpha(\xi,\xi).
	\end{align}
	From parallelism of $\alpha$ and \eqref{2.4}, equation \eqref{3.10} leads to 
	\begin{align*}
	\alpha(X,Y)=-g(X,Y)\alpha(\xi,\xi).
	\end{align*}
	This completes the proof.
\end{proof}
\begin{rem}
	In the above discussion, if $\lambda+\kappa\sigma\neq 0$, then conformally flat perfect fluid spacetime with torse-forming vector field $\xi$ is a regular spacetime.
\end{rem}
In view of Theorem \ref{t3.7}, we state the following corrollary;
\begin{cor}
	A second order symmetric parallel tensor in a conformally flat perfect fluid regular spacetime with torse-forming vector field $\xi$ is a constant multiple of metric tensor $g$.
\end{cor}
From \eqref{2.4}, we have $\nabla_\xi \xi=0$, i.e., the integral curves of $\xi$ are geodesic. Suppose $V$ is an affine killing vector field on perfect fluid spacetime with torse-forming vector field $\xi$, then integrability condition (see \cite{YK}, page 24) $L_V\nabla
=0$, where $L_V$ is the Lie differentiation along $V$. This condition additionally implies that a killing vector field is obviously an affine killing vector field, but, the converse is not necessarily true. \\
Call the following formula; (see \cite{DK} , page 39)
\begin{align}
\label{3.11}(L_V \nabla)(X,Y)=\nabla_X\nabla_Y V-\nabla_{\nabla_X Y} V+R(V,X)Y.
\end{align}
Setting $X=Y=\xi$, in the above formula, one can obtain
\begin{align*}
(L_V \nabla)(\xi,\xi)=\nabla_\xi\nabla_\xi V+R(V,\xi)\xi,
\end{align*}
as $V$ is a affine killing vector field. As a consequence of this, we state the following;
\begin{thm}
	If $V$ is a affine killing vector field on perfect fluid spacetime with tores-forming vector field $\xi$, then $V$ is a Jacobi vector field along the geodesics of $\xi$. 
\end{thm} 
If $\xi$ is a killing vector field, then we have $L_\xi S=0$. This suggest that $(L_\xi Q)X=0$, and gives 
\begin{align*}
0&=L_\xi (QX)-Q(L_\xi X),\\
&=\nabla_\xi QX+\nabla_{QX} \xi-Q(\nabla_\xi X)-Q(\nabla_X \xi),\\
&=(\nabla_\xi Q)X+\nabla_{QX} \xi-Q(\nabla_X \xi).
\end{align*}
Since $\xi$ is a torse-forming vector field, using \eqref{2.4} in the above equation, we have $\nabla_\xi Q=0$. And by \eqref{2.3},  we have $Q\xi=(a-b)\xi$. Next, taking covariant differentiation of this equation, gives
\begin{align*}
(\nabla_X Q)\xi=-\kappa(\sigma+\rho)\nabla_X \xi.
\end{align*}
In view of above discussion, we have a following result
\begin{lem}
	If $\xi$ is a killing vector field in a perfect fluid spacetime with torse-forming vector field $\xi$, then 
	\begin{align*}
	(i)\quad\nabla_\xi Q=0,\qquad(ii)\quad (\nabla_X Q)\xi=-\kappa(\sigma+\rho)\nabla_X \xi.
	\end{align*}
\end{lem}
Now taking $X=QX$ in \eqref{2.2}, we get
\begin{align}
\label{3.12} S(QX,Y)=aS(X,Y)+bS(X,\xi)\eta(Y).
\end{align}
Contracting\eqref{3.12} over $X$ and $Y$, we get
\begin{align}
\label{3.13}
S^2(X,X)=||Q||^2=a\,\,r+b\, S(\xi,\xi).
\end{align}
From \eqref{2.2}, it follows that 
\begin{align}
\label{3.14}S(\xi,\xi)=b-a=-\lambda+\frac{k(\sigma+3\rho)}{2}.
\end{align}
Thus in view of \eqref{2.1} and \eqref{3.14}, \eqref{3.13} yields
\begin{align}
||Q||^2=4\lambda^2+2\lambda\kappa(\sigma-3\rho)+\kappa^2(\sigma^2+3\rho^2).
\end{align}
Hence, we obtain the following theorem;
\begin{thm}
	In a perfect fluid spacetime with torse-forming vector field $\xi$, the square of length of the Ricci operator is $4\lambda^2+2\lambda\kappa(\sigma-3\rho)+\kappa^2(\sigma^2+3\rho^2)$.  
\end{thm}

\section{\bf Ricci Soliton Structure in Perfect Fluid Spacetime with Torse-forming Vector Field}
In this section, we study Ricci soliton structure in a perfect fluid spacetime whose timelike velocity vector field $\xi$ is torse-forming.
\par Now taking $V=\xi$, equation \eqref{1.4} becomes
\begin{align}
\label{4.1} (L_\xi g+2S+2\Lambda g)(X,Y)=0.
\end{align}
In view of \eqref{2.10}, we have
\begin{align*}
S(X,Y)+(\Lambda+1)g(X,Y)+\eta(X)\eta(Y)=0.
\end{align*}
By using \eqref{2.2} in the above realtion, it gives
\begin{align}
\label{4.2}(a+1+\Lambda)g(X,Y)+(b+1)\eta(X)\eta(Y)=0.
\end{align}
On plugging $X=Y=\xi$ in \eqref{4.2}, we get
\begin{align}
\label{4.3} \Lambda=\frac{\kappa}{2}(\sigma+3\rho)-\lambda.
\end{align}
Thus, we have
\begin{thm} \label{t4.11}
	If a perfect fluid spacetime with torse-forming vector field $\xi$ admits a Ricci soliton $(g,\xi,\Lambda)$, then Ricci soliton is expanding, steady and shrinking according as 
$\frac{\kappa}{2}(\sigma+3\rho)>\lambda$, $\frac{\kappa}{2}(\sigma+3\rho)=\lambda$ and $\frac{\kappa}{2}(\sigma+3\rho)<\lambda$ respectively.
\end{thm}
	Suppose we consider the given spacetime as a spacetime without cosmological constant i.e., $\lambda=0$. Then from \eqref{3.14}, it gives $S(\xi,\xi)=\frac{\kappa(\sigma+3\rho)}{2}$. If the given spacetime satisfies the timelike convergence condition i.e., $S(\xi,\xi)>0$, then $\sigma+3\rho>0$, the spacetime obeys cosmic strong energy condition.
	
	 In view of above discussion and from equation \eqref{4.3}, we conclude that

\begin{thm}
	If a perfect fluid spacetime with torse-forming vector field $\xi$ without cosmological constant admits Ricci soliton $(g,\xi,\Lambda)$ and satisfies timelike convergence condition, then Ricci soliton is expanding.
\end{thm}

Making use of \eqref{2.2}, the soliton equation \eqref{1.4} takes the form
\begin{align}
\label{4.4} (L_V g)(X.Y)=-2\{(a+\Lambda)g(X,Y)+b\eta(X)\eta(Y)\}.
\end{align}
Taking Lie-differentiation of \eqref{2.2} along the vector field $V$ and using \eqref{4.4}, it gives
\begin{align}
\label{4.5} (L_V S)(X,Y)=&b\{(L_V \eta)(X)\eta(Y)+\eta(X)(L_V\eta)(Y)\}\\
\nonumber &-2a\{(a+\Lambda)g(X,Y)+b\eta(X)\eta(Y)\}.
\end{align}
On the other hand, differentiating \eqref{2.2} covariantly along $Z$ and then using \eqref{2.5}, we obtain
\begin{align}
\label{4.6}
(\nabla_Z S)(X,Y)=b\{g(Z,X)\eta(Y)+g(Z,Y)\eta(X)+2\eta(X)\eta(Y)\eta(Z)\}.
\end{align}
Using the soliton equation \eqref{1.4} in commutation formula (\cite{YK}, page 23)
\begin{align*}
(L_V\nabla_Z g-\nabla_ZL_V g-\nabla_{[V,Z]})(X,Y)=-g((L_V\nabla)(Z,X),Y)-g((L_V\nabla)(Z,Y),X),
\end{align*}
we get
\begin{align}
\label{4.7} g((L_V\nabla)(X,Y),Z)=(\nabla_Z S)(X,Y)-(\nabla_X S)(Y,Z)-(\nabla_Y S)(X,Z)
\end{align}
In view of \eqref{4.6}, equation \eqref{4.7} takes the form
\begin{align}
\label{4.8}(L_V\nabla)(X,Y)=-2b\{g(X,Y)\xi+\eta(X)\eta(Y)\xi\}.
\end{align}
Taking covariant differentiation of \eqref{4.8} along Z and using \eqref{2.4} and \eqref{2.5}, yields
\begin{align}
\label{4.9} (\nabla_Z L_V\nabla)(X,Y)=&-2b\{g(X,Y)Z+\eta(X)\eta(Y)Z+g(X,Y)\eta(Z)\xi\\
\nonumber &+g(Z,X)\eta(Y)\xi+g(Z,Y)\eta(X)\xi+3\eta(X)\eta(Y)\eta(Z)\xi\}.
\end{align}
Again, according to Yano \cite{YK} we have the following commutation formula
\begin{align}
\label{4.10} (L_V R)(X,Y)Z=(\nabla_X L_V \nabla)(Y,Z)-(\nabla_Y L_V\nabla)(X,Z).
\end{align}
In view of \eqref{4.9}, equation \eqref{4.10} takes the form
\begin{align}
\label{4.11} (L_V R)(X,Y)Z=2b\{g(X,Z)Y-g(Y,Z)X+\eta(X)\eta(Z)Y-\eta(Y)\eta(Z)X\}.
\end{align}
On contracting \eqref{4.11}, we obtain
\begin{align}
\label{4.12}
(L_V S)(Y,Z)=-6b\{g(Y,Z)+\eta(Y)\eta(Z)\}.
\end{align}
Taking $Y=Z=\xi$ in \eqref{4.12}, we have
\begin{align}
\label{4.13} (L_V S)(\xi, \xi)=0.
\end{align}
Substituting $X=Y=\xi$ in \eqref{4.5} and then using \eqref{4.13}, we get 
\begin{align}
\label{4.14} -2b(L_V\eta)(\xi)+ 2a(a+\Lambda-b)=0.
\end{align}
Setting $Y=\xi$ in \eqref{4.4} it follows that $(L_Vg)(X,\xi)=-2(a+\Lambda-b)\eta(X)$. Lie-differentiating the equation $\eta(X)=g(X,\xi)$ along $V$ and by virtue of last equation, we find
\begin{align}
\label{4.15} (L_V \eta)(X)-g(L_V \xi,X)+2(a+\Lambda-b)\eta(X)=0.
\end{align}
Next, Lie-differentiating $g(\xi,\xi)=-1$ along $V$ and then using \eqref{4.4}, we obtain
\begin{align}
\label{4.16}\eta(L_V \xi)=-(a+\Lambda-b).
\end{align}
In view of \eqref{4.15} and\eqref{4.16}, we obtain
\begin{align}
\label{4.17}(L_V \eta)(\xi)=-\eta(L_V \xi)=(a+\Lambda-b).
\end{align}
Using \eqref{4.17} in \eqref{4.14}, we get
\begin{align}
(2a-2b)(a+\Lambda-b)=0.
\end{align}
After substituting the value of $a$ and $b$, we find
\begin{align}
[2\lambda-\kappa(\sigma+3\rho)]\left\{\lambda-\frac{\kappa(\sigma+3\rho)}{2}+\Lambda\right\}=0.
\end{align}
This implies either $\lambda=\frac{k(\sigma+3\rho)}{2}$, or $\Lambda=\frac{\kappa(\sigma+3\rho)}{2}-\lambda$. Thus we have the following cases;\\
{\bf Case 1:} If $\lambda=\frac{k(\sigma+3\rho)}{2}$, then $\Lambda\neq0$ and it follows that Ricci soliton is not steady. In this case equation \eqref{2.2} takes the form 
\begin{align}
S(X,Y)=\kappa(\sigma+\rho)\{g(X,Y)+\eta(X)\eta(Y)\}
\end{align}
i.e., perfect fluid spacetime with torse-forming vector field $\xi$ is a spacetime with the equal associated scalar constant.\\
{\bf Case 2:} If $\lambda\neq \frac{k(\sigma+3\rho)}{2}$, then $\Lambda=\frac{\kappa(\sigma+3\rho)}{2}-\lambda$. In this case Ricci soliton is expanding (or shrinking) according as $\frac{\kappa(\sigma+3\rho)}{2}>\lambda$ ( or $\frac{\kappa(\sigma+3\rho)}{2}<\lambda$).\\
Thus, we have
\begin{thm}
	If a perfect fluid spacetime with torse-forming vector field $\xi$ admits a Ricci soliton $(g,V,\Lambda)$, then either every perfect fluid spacetime with torse-forming vector field $\xi$ is a spactime with the equal associated scalar and Ricci soliton is not steady, or the Ricci soliton is expanding (or shrinking) according as $\frac{\kappa(\sigma+3\rho)}{2}>\lambda$ (or $ \frac{\kappa(\sigma+3\rho)}{2}<\lambda$). 
\end{thm}

If $X$ and $Y$ are replaced with $\xi$, then it follows from \eqref{4.8} that 
\begin{align}
\label{4.21} (L_V \nabla)(\xi,\xi)=0.
\end{align}
Recalling the formula \eqref{3.11} togther with $X=Y=\xi$ and then using \eqref{2.4} and \eqref{4.21}, we conclude that
\begin{align}
\label{4.22}\nabla_\xi\nabla_\xi V+R(V,\xi)\xi=0.
\end{align}
Hence, equation \eqref{4.22} implies that potential vector field $V$ is a Jacobi vector field along the geodesics of $\xi$. Thus we formulate the following theorem;
\begin{thm}
	If a perfect fluid spacetime with torse-forming vector field $\xi$ admits a Ricci soliton togther with the potential vector field $V$, then $V$ is a  Jacobi vector field along the geodesics of $\xi$.
\end{thm}
 Equation \eqref{4.4} can be written as 
\begin{align}
\label{4.23}g(\nabla_X V,Y)+g(\nabla_Y V,X)+2\{(a+\Lambda)g(X,Y)+b\eta(X)\eta(Y)\}=0.
\end{align}
Suppose $\omega$ is $1-$form, metrically equivalent to $V$ and is given by $\omega(X)=g(X,V)$ for an arbitrary vector field $X$, then the exterior derivative $d\omega$ of $\omega$ is given by
\begin{align}
\label{4.24}2(d\omega)(X,Y)=g(\nabla_X V, Y)-g(\nabla_YV,X).
\end{align}
As $d\omega$ is a skew-symmetric, if we define a tensor field $F$ of type $(1,1)$ by
\begin{align}
(d\omega)(X,Y)=g(X,FY).
\end{align}
then, $F$ is skew self-adjoint i.e., $g(X,FY)=-g(FX,Y)$. Thus equation \eqref{4.24} takes the form
\begin{align*}
2g(X,FY)=g(\nabla_X V, Y)-g(\nabla_YV,X).
\end{align*}
Adding it to equation \eqref{4.23} side by side, and factoring out $Y$ gives
\begin{align}
\label{4.26}\nabla_X V=-FX-(a+\Lambda)X-b\eta(X)\xi.
\end{align}
Substituting this equation in $R(X,Y)V=\nabla_X\nabla_Y V-\nabla_Y \nabla_X V-\nabla_{[X,Y]} V$ we obtain
\begin{align}
\label{4.27}R(X,Y)V=(\nabla_Y F)X-(\nabla_X F)Y+b\{Y\eta(X)-X\eta(Y)\}.
\end{align}
Noting that $d\omega$ is closed, we have
\begin{align}
\label{4.28}g(X,(\nabla_Z F)Y)+g(Y.(\nabla_X F)Z)+g(Z,(\nabla_Y F)X)=0.
\end{align}
Taking inner product of \eqref{4.27} with $Z$ we get
\begin{align}
\label{4.29}g(R(X,Y)V,Z)=&g((\nabla_Y F)X,Z)-g((\nabla_X F)Y,Z)\\
\nonumber&+b\{g(Y,Z)\eta(X)-g(X,Z)\eta(Y)\}
\end{align}
The skew self-adjointness of $F$ implies skew self-adjointness of $\nabla_X F$. Thus using \eqref{4.28} in \eqref{4.29} gives
\begin{align}
\label{4.30}g(R(X,Y)V,Z)=b\{g(Y,Z)\eta(X)-g(X,Z)\eta(Y)\}-g(X,(\nabla_Z F)Y).
\end{align}
Setting $X=Z=\sum_{i=1}^{4}e_i$ in \eqref{4.30}, where $e_i$'s be a local orthonormal frames, provides
\begin{align}
\label{4.31} S(Y,V)=3b\eta(Y)-(div F)Y,
\end{align}
where $div F$ denotes divergence of tensor field $F$. Equating \eqref{4.31} with \eqref{2.2}, gives
\begin{align}
\label{4.32}(div F)Y=\kappa(\sigma+\rho)(3-\eta(V))\eta(Y)-\left\{\lambda+\frac{\kappa(\sigma-\rho)}{2}\right\}\omega(Y).
\end{align}
We compute the covariant derivative of the squared $g$-norm of V using
\eqref{4.26} as follows
\begin{align}
\label{4.33}\nabla_X|V|^2=2g(\nabla_XV.V)=-2 g(FX,V)-2\{(a+\Lambda)g(X,V)+b\eta(X)\eta(V)\}.
\end{align}
In view of \eqref{4.4}, equation \eqref{4.33} takes the form
\begin{align}
\label{4.34}\nabla_X|V|^2+2g(FX,V)+2(L_V g)(X,V)=0.
\end{align}
Hence, we state the following 
\begin{thm}
	If a perfect fluid spacetime with torse-forming vector field $\xi$ admits a Ricci soliton, then Ricci soliton vector $V$ and its metric dual $1-$form $\omega$ satisfies the formula \eqref{4.32} and \eqref{4.34}.
\end{thm}

{\bf Authors address:}\\
{\bf Venkatesha, Aruna Kumara H}\\
Department of Mathematics, Kuvempu University,\\
 Shankaraghatta - 577 451, Shimoga, Karnataka, INDIA.\\
 e-mail: {\verb+vensmath@gmail.com+, arunmathsku@gmail.com}\\
\end{document}